\documentclass[12pt]{article}
\usepackage{amssymb, amsmath, amsthm}
\usepackage{fullpage}
\newtheorem{theorem}{Theorem}
\newtheorem{lemma}{Lemma}
\newtheorem{corollary}{Corollary}

\newtheorem{conjecture}{Conjecture}

\newtheorem{proposition}{Proposition}

\newtheorem{question}{Question}

\def\beq{\begin{equation}}\def\eeq{\end{equation}}
\def\beqn{\begin{eqnarray}}\def\eeqn{\end{eqnarray}}

\def\qed{\ifhmode\unskip\nobreak\fi\quad\ifmmode\Box\else$\Box$\fi}

\newcommand{\x}{\underline{x}}
\newcommand{\y}{\underline{y}}
\newcommand{\e}{\underline{1}}
\newcommand{\diff}{\mathrm{diff}}

\title{Around a biclique cover conjecture}
\author{Guantao Chen\\
\small Department of Mathematics and Statistics\\[-1ex]
\small Georgia State University\\[-1ex]
\small Atlanta, Georgia 30303\\[-1ex]
\small \texttt{matgtc@langate.gsu.edu}\\[-0.4ex]\\
\and \hspace*{100pt}Shinya Fujita\hspace*{100pt}\\
\small Department of Integrated Design Engeneering,\\[-1ex]
\small Maebashi Institute of Technology,\\[-1ex]
\small 460-1 Kamisadori, Maebashi 371-0816, Japan\\[-1ex]
\small \texttt{shinyaa@mti.biglobe.ne.jp}\\[-0.4ex]\\
\and Andr\'as Gy\'arf\'as\\
\small Computer and Automation Research Institute\\[-1ex]
\small Hungarian Academy of Sciences\\[-1ex]
\small 1518 Budapest, P.O. Box 63, Hungary\\[-1ex]
\small \texttt{gyarfas@sztaki.hu}\\[-0.4ex]\\
\and Jen\H{o} Lehel\\
\small Department of Mathematical Sciences\\[-1ex]
\small The University of Memphis, Tennessee, USA\\[-1ex]
\small \texttt{jlehel@memphis.edu}\\[-0.4ex]\\
\and \'Agnes T\'oth\\
\small Department of Computer Science and Information Theory\\[-1ex]
\small Budapest University of Technology and Economics\\[-1ex]
\small 1521 Budapest, P.O. Box 91, Hungary\\[-1ex]
\small \texttt{tothagi@cs.bme.hu}\\[-0.4ex]\\[2.2ex]\\}
\begin{document}
 \maketitle
\newpage
\begin{abstract}
We address an old (1977) conjecture of a subset of the authors (a
variant of Ryser's conjecture): in every $r$-coloring of the edges
of a biclique $[A,B]$ (complete bipartite graph), $A\cup B$ can be
covered by the vertices of at most $2r-2$ monochromatic connected
components. We reduce this conjecture to design-like conjectures,
where the monochromatic components of the color classes are
bicliques $[X,Y]$ with nonempty blocks $X$ and $Y$.
It can be also assumed that each color class covers $A\cup B$
(spanning), moreover, no $X$-blocks or $Y$-blocks properly
contain each other (antichain property). We prove this reduced
conjecture for $r\le 5$.

We show that the width (the number of bicliques) in every color
class of any spanning $r$-coloring is at most $2^{r-1}$ (and this is
best possible).
On the other hand there exist spanning $r$-colorings such that the width of every color class is $\Omega(r^{3/2})$.

We discuss the dual form of the conjecture which relates to
transversals of intersecting and cross-intersecting $r$-partite
hypergraphs.
\end{abstract}

\section{Introduction, summary of results.}

A special case of a conjecture generally attributed to Ryser
(appeared in his student, Henderson's thesis, \cite{HE}) states that
intersecting $r$-partite hypergraphs have a transversal of at most
$r-1$ vertices (see Conjecture \ref{Ryshyp} in Section \ref{dual}).
This conjecture is open for $r\ge 6$. It is  trivially true
for $r=2$,  the cases
$r=3,4$  are solved in \cite{GY} and in \cite{DU}, and for the case $r=5$, see
 \cite{DU} and  \cite{TU}.
The following equivalent formulation is from \cite{GY},\cite{GYEP}:

\begin{conjecture}\label{rys}
\label{Rysconn}
In every $r$-coloring of the edges of a
complete graph, the vertex set can be covered by the vertices of at most $r-1$
monochromatic connected components.
\end{conjecture}

Gy\'arf\'as and Lehel discovered a bipartite version of this conjecture
\cite{GY}, \cite{LEH}. A complete bipartite graph $G$ with non empty vertex classes $X$ and $Y$ is referred to as a {\it biclique} $[X,Y]$ in
this paper, and $X$ and $Y$ will be called the {\it blocks} of this biclique. Given an edge coloring, a
{\it monochromatic component} means a connected
component of the subgraph of any given color. The number of components in a given color is called the {\em width of the color.}

\begin{conjecture}\label{biprys}
\label{bipconn}
In every $r$-coloring of the edges of a
biclique, the vertex set can be covered by the vertices of at most $2r-2$
monochromatic components.
\end{conjecture}

First we show here that Conjecture \ref{bipconn}, if true, is best
possible. Let $G^*=[A,B]$ be a biclique with $|A|=r-1,|B|=r!$, and
label the vertices of $A$ with $\{1,2,\dots,r-1\}$ and those of $B$
with the $(r-1)$-permutations of the elements of $\{1,2,\dots,r\}$. For $k\in A$ and $\pi=j_1j_2\dots j_{r-1}\in B$, let the color of the
edge $k\pi$ be $j_k$.

Since each vertex in $B$ is incident with $r-1$ edges of distinct color, every monochromatic component of $G^*$ is a star with $(r-1)!$ leaves centered at $A$.
Furthermore,  $G^*$ has a vertex cover with $2r-2$ monochromatic components, just take the $r$ monochromatic stars centered at vertex $r-1$, and add one edge from each vertex $k=1,2,\dots,r-2$ of $A$.

\begin{proposition}
\label{strange} {\rm (\cite{GY})} The vertex set of $G^*$ cannot be covered with
less than $2r-2$ monochromatic components.
\end{proposition}

\begin{proof} Let $\cal{C}$ be a cover of $V(G^*)=A\cup B$ by monochromatic stars
centered in $A$. Let $a_k$ denote the number of monochromatic
stars of $\cal{C}$ on vertex $ k\in A$. We
may assume that $a_1\le a_2\le \dots \le a_{r-1}$.

We show first that $a_i\ge i+1$ holds for some $1\leq i\leq r-1$.
Suppose on the contrary that $a_i\leq i$, for all $i$.
Thus we can select a color $j_{r-1}\in\{1,\dots,r\}$ different from
the $a_{r-1}$ colors of all stars of $\cal{C}$ centered at $r-1$.
Then we can select a new color $j_{r-2}\in\{1,\dots,r\}\setminus
\{j_{r-1}\}$ different from the $a_{r-2}$ colors of all stars of
$\cal{C}$ centered at $r-2$, etc. Thus we end up by selecting
$r-1$ distinct colors $j_1,\dots,j_{r-1}$. This is a contradiction
since the $(r-1)$-permutation $j_1j_2\dots, j_{r-1}\in B$ is
uncovered by $\cal{C}$.

Now let  $a_i\ge i+1$, for some $1\leq i\leq r-1$, then the number of
stars in $\cal{C}$ is
$$\sum_{k=1}^{r-1} a_k=\sum_{k=1}^{i-1} a_k+\sum_{k=i}^{r-1} a_k\ge
(i-1)+(i+1)(r-i).$$
Because  $$(i-1)+(i+1)(r-i)=-i^2+ri +r-1\ge 2r-2$$ holds for every $1\leq i\leq
r-1$, the proposition  follows.
\end{proof}

It is worth noting that Conjecture \ref{biprys} (similarly to
Conjecture \ref{rys}) becomes obviously true if the number of
monochromatic components is just one larger than stated in the
conjecture.

\begin{proposition}
\label{2r-1} {\rm (\cite{GY})}
In every $r$-coloring of the edges of a
biclique, the vertex set can be covered by the vertices of at most $2r-1$
monochromatic components.
\end{proposition}
\begin{proof}
For an edge $uv$ of the biclique $G$, consider the monochromatic
component (double star) formed by the edges in the color of $uv$
incident to $u$ or $v$. In all other colors consider the
(at most $r-1$) monochromatic stars centered at $u$ and at $v$. This
gives $2r-1$ monochromatic components covering  the vertices of
$G$. \end{proof}

In Section \ref{equi} we show that Conjecture \ref{biprys} can be
reduced to design-like conjectures: one can assume that all colors span {\em bi-equivalence graphs}, i.e. graphs whose
components are complete bipartite graphs. It is worth noting that a similar reduction is not known for Conjecture
\ref{rys}.

Our results in Sections \ref{234} and \ref{5} (Corollary \ref{23width}, Theorems \ref{r=4}, \ref{r=5}) imply

\begin{theorem}\label{smallr} Conjecture \ref{biprys} holds for $r= 2,3,4,5$.
\end{theorem}

Suppose we have a partition of a biclique into bi-equivalence graphs. We call a pair $u,v$ in one of the cliques of a biclique  {\it equivalent} if $u$ and $v$ belong to the same block in every bi-equivalence graph.  Since equivalent vertices do not change the number of components needed for a cover, the following result shows that for every fixed $r$ one has to consider only finitely many colorings.

\begin{theorem} \label{rfact} Suppose a biclique $[A,B]$ is partitioned into
$r$ bi-equivalence graphs and no two vertices of $A$ are
equivalent. Then $\max\{|A|,|B|\}\le r!$
and equality is possible.
\end{theorem}

It is natural to ask how many monochromatic components (or bicliques) {\em of the same color}  cover all vertices in $r$-colorings of bicliques, i.e. to bound the minimum width of the color classes. Such coverings are called {\em homogeneous} in Section \ref{hom}. In the example $G^*$ used in Proposition \ref{strange} the width of every color class is $(r-1)!+r-1$ (this property of $G^*$ played a role in \cite{GYRSSZ} where coverings by monochromatic cycles have been studied). However, for {\em spanning colorings}, where at least one edge is adjacent in each color to any vertex, the situation is different: using a deep result of Alon \cite{AL}, we show

\begin{theorem}\label{main}
\label{main} In a spanning $r$-coloring the width of every color class is at most $2^{r-1}$ and this is best possible.
\end{theorem}

It is tempting to conjecture (in fact one of the authors did) that Conjecture \ref{biprys} is true in a stronger form: {\em some} color class in every spanning $r$-coloring has width at most $2r-2$. However,

\begin{theorem}\label{widthlower} There are spanning $r$-partitions of
bicliques such that the width of every partition class is $\Omega(r^{3/2})$.
\end{theorem}

Theorem \ref{widthlower} naturally suggests the following question.

\begin{question}\label{homrcon}
\label{rmonohomo} Determine or estimate $g(r)$, the largest $m$ such that there is a spanning partition
of a biclique into $r$ bi-equivalence graphs, all with width at least $m$.
\end{question}

Very recently T. Terpai \cite{TER} improved the bound of Theorem \ref{widthlower} to $g(r)=\Omega(r^2)$.

\medskip

In Section \ref{dual} we formulate the dual forms of Conjectures
\ref{rys}, \ref{biprys} and show their relation to transversals of
intersecting and cross-intersecting $r$-uniform hypergraphs.

\section{Equivalent conjectures, notations.}\label{equi}

Here we prove some equivalent forms of Conjecture \ref{bipconn} leading to a
design-like conjecture (Conjecture \ref{anticonj}). In this spirit an $r$-coloring will be also called a partition
of the edge set into $r$ subgraphs.

\medskip

\noindent {\bf A.} {\it If a biclique is partitioned into $r$
bi-equivalence graphs, then its vertex set can be covered by at most $2r-2$
biclique components.}\\

Since the bi-equivalence graphs in claim A can be color
classes of an $r$-coloring, validity of Conjecture \ref{bipconn} implies that
claim A is also true.

On the other hand, suppose we have an $r$-coloring of a biclique
$G=[X,Y]$  such that some monochromatic
component $C$, say in color $1$, is not a biclique.
 Let $x\in X,y\in Y$ be non-adjacent vertices in $C$, w.l.o.g. $xy$ has
color $2$. Observe that the $2(r-2)$ monochromatic stars in colors
$3,\dots, r$ centered at $x$ and at $y$, plus the component $C$, and
the component in color $2$ containing $xy$ cover $V(G)$, leading to
a cover with at most $2r-2$ monochromatic components. Thus
Conjecture \ref{bipconn}  follows from claim A.\\

Let us call a bi-equivalence graph partition $G_1,\dots,G_r$ of biclique $G$ a
\it spanning partition \rm if each vertex $v\in V(G)$ is included in every
$V(G_i)$,  $i=1,\dots,r.$
Notice that it is enough to prove claim A for spanning partitions. Indeed, assuming that $v\not\in V(G_1)$
and $vw\in E(G_2)$, just take the at most $r-2$ bicliques from  $G_3,\dots,G_r$
that contain $v$ and add the at most $r$ bicliques  from $G_1,G_2,\dots,G_r$
that contain $w$, together they form a cover of all vertices of $G$ with at most
$2r-2$ bicliques. Thus we have the following equivalent form of claim A:\\

\noindent {\bf B.} {\it If a biclique has a spanning partition into $r$
bi-equivalence graphs, then its vertex set can be covered by at most
$2r-2$ biclique components.}\\

Let a biclique $[X,Y]$ be partitioned into the bi-equivalence graphs
$G_1, G_2, \dots, G_r$. Then we will say that $i$ is the color of
the edges in $G_i$ $(i=1,\dots,r)$. Any connected component of $G_i $
is a biclique, its vertex classes will called
{\it blocks in color $i$}.\\

Denote by $B_i[u_1, \dots, u_k]$ the connected component of $G_i$
which contains the vertices $u_1, \dots, u_k$, if they are in the
same component of $G_i$, and in this case let $X_i[u_1, \dots,
u_k]=X\cap V(B_i[u_1, \dots, u_k])$ and $Y_i[u_1, \dots, u_k]=Y\cap
V(B_i[u_1, \dots, u_k])$ be the corresponding blocks. Otherwise set
$B_i[u_1, \dots, u_k]=\emptyset$, $X_i[u_1, \dots, u_k]=Y_i[u_1,
\dots, u_k]=\emptyset$.

Note that $B_i[u]\neq\emptyset$ for any $u\in V(G)$ in a spanning partition. In
the sequel we will also use the fact that the blocks $X_i[u]$ and $X_i[v]$ are either disjoint or
equal for any color $i\in\{1,2,\dots, r\}$ and any vertices $u,v\in V(G)$.\\

Let us call a spanning bi-equivalence graph partition $G_1, \dots, G_r$ of
biclique $G$ an \emph{antichain partition} if no two blocks properly
contain each other, that is if no colors $i, j\in \{1,\dots, r\}$
and no vertices $u, v\in V(G)$ exist such that $X_i[u]\subsetneq
X_j[v]$ or $Y_i[u]\subsetneq Y_j[v]$.

If $v\in X$ and $|X_i[v]|=1$ (or $v\in Y$ and $|Y_i[v]|=1$) then we call vertex $v$ a
{\it singleton} block in color $i$. Note that if a coloring has the antichain property, then a
singleton block in some color is a singleton in every color, in this case we just say that $v$ is
a singleton.\\

It turns out that it is enough to prove claim B
for antichain partitions. Indeed, assume that in
a spanning partition  there are two blocks properly containing each
other, that is $X_1[z]\subsetneq X_2[x]$, for some biclique
components $B_1[z]$ and $B_2[x]$. Assume that $x\notin X_1[z]$, and
let $y\in Y_1[z]$. The color of the edge $xy$ is neither $1$ nor
$2$, w.l.o.g. it is $3$. Because $B_3[y]=B_3[x]$ and
$X_1[y]=X_1[z]\subseteq X_2[x]$,
 the collection
$$\{ B_i[x]\ :\ i\in\{ 1,2,\dots, r\}\}\cup
\{ B_i[y]\ :\ i\in\{ 1,2,\dots, r\}\setminus\{1,3\}\}$$ is a cover
with at most $2r-2$ monochromatic components. Thus we obtain the
following equivalent form of Conjecture \ref{biprys}.

\begin{conjecture}\label{anticonj}
If a biclique has an antichain partition into $r$ bi-equivalence
graphs, then its vertex set can be covered by at most $2r-2$
biclique components.
\end{conjecture}

Finally we note an important reduction used extensively in the proofs
later. We recall that a pair $u,v\in A$ or $u,v\in B$ {\it equivalent} if in every bi-equivalence graph of the
bi-equivalence graph partition of the biclique $G$, $u$ and $v$ belong to the same block.  We may assume w.l.o.g. that there is no pair of equivalent vertices, and in this case we say that {\it the
coloring is reduced}. Indeed, if there were two vertices $u,v$ such
that $uv\notin E(G)$ and for every $w\in V(G)$ with $uw, vw\in
E(G)$, the edges $uw$ and $vw$ have the same color, then $v$ could
be added to any monochromatic component of $G-\{v\}$ containing $u$.
Hence if Conjecture \ref{anticonj} holds for $G-\{v\}$ then it also
holds for $G$.

\medskip

\noindent {\textbf{Theorem \ref{rfact}.}} {\em Suppose a biclique $[A,B]$ has a partition into
$r$ bi-equivalence graphs and no two vertices of $A$ are
equivalent. Then $\max\{|A|,|B|\}\le r!$
and equality is possible.}

\begin{proof}It is easy to check that the partition of $G^*$ into bi-equivalence graphs
in Proposition \ref{strange} is a reduced one, hence the second
statement follows.

To prove the first statement, the case $r=1$ is obvious. Assuming it
is true for some $r\ge 1$, suppose indirectly that $|A|\ge (r+1)!+1$
in some partition into $r+1$ bi-equivalence graphs. Then for any
fixed $v\in B$ there are $r!+1$ edges of the same color from $v$,
say in color $r+1$, to $Y\subset A$. Let $X$ be the set of vertices
in $B$ that send edges in at least two different colors to $Y$. By
the assumption $X\ne \emptyset$ and since color class $r+1$ is a
bi-equivalence graph, $[X,Y]$ has no edge of color $r+1$. This means
no two vertices of $Y$ are equivalent in the induced $r$-partition
on $[X,Y]$, and thus $|Y|>r!$ contradicts the inductive hypothesis.
\end{proof}

\section{Homogeneous covering.}\label{hom}
\label{homo} In 1998 Guantao Chen asked whether a stronger version
of claim B can be true, i.e. whether $2r-2$
biclique components \it of the same bi-equivalence graph \rm
$G_i$, $1\leq i\leq r,$ can cover $[X,Y]$. Call such cover a {\it
homogeneous cover}. Although this is not true in general (see Theorem \ref{widthlower} below), the question introduces interesting
variants of the cover problem.

Given $r$, let $g(r)$ be the smallest $m$ such that in every biclique $B$ with a
spanning partition into $r$ bi-equivalence graphs  $G_1,\dots,G_r$, there is a
partition class $G_i$ with width at most $m$. We shall prove that $g(r)$ exists,
in a stronger form: for every $r$, there is a smallest $m=h(r)$ such that in
every spanning partition of a biclique into $r$ bi-equivalence graphs, the
width of \it every partition class \rm is at most $m$.

\medskip

\noindent {\textbf{Theorem \ref{main}.}} {\em $h(r)=2^{r-1}$.}

\begin{proof} To see that $h(r)\ge 2^{r-1}$ consider the following
easy recursive construction to partition a biclique into $r$ bi-equivalence
graphs such that the maximum width is $2^{r-1}$. The case $r=1$ is obvious.
Given such a spanning partition of $B=K_{n,n}$ into $r$ bi-equivalence classes,
take two vertex disjoint copies of $B$ and place two bicliques crosswise as the
$r+1$-th partition. This way a spanning partition of $K_{2n,2n}$ is obtained
into $r+1$ bi-equivalence graphs and the width of every partition class is
doubled - apart from the $(r+1)$-th class which has width two.

To prove the other direction, $h(r)\le 2^{r-1}$, we need some
definitions. An equivalence graph is a graph whose components are
complete graphs. Let $eq(G)$ denote the minimum number of spanning
equivalence graphs needed to cover the edge set of a graph $G$.
Similarly, for any bipartite graph $G$, let $eqbi(G)$ denote the
minimum number of spanning bi-equivalence graphs needed to cover the
edges of $G$. Let $G^+$ denote the graph obtained from the bipartite
graph $G$ by adding to $E(G)$ all pairs inside the partite classes
of $G$. Let $K_{t,t}^-=K_{t,t}-tK_2$, i.e. $K_{t,t}^-$ is a balanced
biclique from which a perfect matching is removed. We need the next
two straightforward propositions.

\begin{proposition}\label{eqbi}
For any bipartite graph $G$, $eqbi(G)\ge eq(G^+)-1$.\end{proposition}
\begin{proof} Consider an optimal cover of $E(G)$ with $eqbi(G)$
spanning  bi-equivalence graphs and turn them into spanning
equivalence graphs by adding all missing edges to all biclique
components. These plus one more spanning equivalence graph formed by
the two vertex classes of $G$ cover all edges of $G^+$ thus
$eqbi(G)+1$ is an upper bound of $eq(G^+)$.
\end{proof}

\begin{proposition}\label{teqbi}
If $B$ is a biclique and $G=B-E(H)$, where $H$ is a spanning bi-equivalence
subgraph of $B$ with $t\geq 2$ components, then $eqbi(G)=eqbi(K_{t,t}^-)$.
\end{proposition}
\begin{proof}
Suppose $X_i,Y_i$ are the bicliques of $H$, $i=1,2 ,\dots,t$ and
$x_iy_i$ are the removed edges of $K_{t,t}$.

If $\{H_l: 1\le l \le s\}$ is a spanning partition of $K_{t,t}^-$
into bi-equivalence graphs, define $G_l$ by adding all edges of
all bipartite graphs $[X_i,Y_j]$ whenever $x_iy_j$ is an edge of
biclique of $H_l$. This defines $\{G_l: 1\le l \le s\}$ as a
spanning partition of $G$ into bi-equivalence graphs showing that
$eqbi(G)\le s$.

To see the reverse inequality, consider an arbitrary cover of $G$
by spanning bi-equivalence graphs $G_1,\dots,G_k$.
Let $T$ be the subset of $2t$ vertices of $V(G)$ containing one vertex from each partite class of each bipartite component of $H$.
For any $1\le l \le k,$ define $H_l$ as the induced subgraph of $G_l$ on $T$.
Then $\{H_l: 1\le l \le
k\}$ is a spanning partition of $K_{t,t}^-$ into bi-equivalence
graphs showing that $k \ge eqbi(K_{t,t}^-)$.
\end{proof}

The main tool is the following result of Alon \cite{AL}.

\medskip

\noindent {\textbf{Theorem (\cite{AL}).}} {\em Suppose that the maximum degree of the
complement of a graph $G$ is $d$ and $|V(G)|=n$. Then $eq(G)\ge
\log_2n-\log_2d$.}

Suppose indirectly that $B$ is a biclique with a spanning partition into
bi-equivalence graphs $G_1,\dots, G_r$ such that some of them, say $G_1$ has
width $t>2^{r-1}$. Let $G=B-G_1$. Using Propositions \ref{eqbi}, \ref{teqbi}
and Alon's theorem, we obtain that
$$eqbi(G)= eqbi(K_{t,t}^-)\ge eq((K_{t,t}^-)^+)-1\ge
\log_2(2t)-\log_21-1>\log_2(2^r)-1=r-1$$ which is a contradiction
since the $r-1$ bi-equivalence graphs $G_2,\dots, G_r$  partition
$G=B-G_1$. Consequently $t\leq 2^{r-1}$, and  $h(r)=
2^{r-1}$follows. This concludes the proof of Theorem \ref{main}.
\end{proof}

Next we prove Theorem \ref{widthlower}, a lower bound for $g(r)$.

\medskip

\noindent {\textbf{Theorem \ref{widthlower}.}}
{\em There are spanning $r$-partitions of
bicliques such that the width of every partition class is $\Omega(r^{3/2})$}.

\begin{proof} Let $s\ge 3$ be an integer, set $r=(s-2)s$ and
$p={s-1\choose 2}$. We shall construct a spanning $r$-partition of the biclique
$K_{sp,sp}$ into bi-equivalence graphs such that each class will be the disjoint
union of one copy of the biclique $K_{p,p}$ and $s-1$ copies of the matching
$pK_2$. Notice that each of those $r$ classes has width $p(s-1)+1\geq c
r^{3/2},$ with constant $c$.

The construction is as follows. Let us color the edges of a Hamiltonian cycle of
$K_{s,s}$ red, and all the other edges of $K_{s,s}$  blue. Each of the
$s^2-2s=r$ blue edges can be uniquely extended with $s-1$ red edges into a
$1$-factor of $K_{s,s}$. Therefore, each red edge belongs to the same number,
$r(s-1)/2s=p$ such $1$-factors. Now we replace each vertex by a set of $p$
elements, every blue edge with a copy of $K_{p,p}$,  and every red edge with $p$
pairwise disjoint copies of $pK_2$.
 \end{proof}

The lower bound of Theorem \ref{widthlower} is recently improved by T. Terpai, \cite{TER}. In fact his construction is not only spanning but also an antichain partition. What we know about the functions $g$ and $h$ is $\Omega(r^2)=g(r)\le
h(r)=2^{r-1}$, and it is a challenging question how they separate.\\

\section{Bi-equivalence partitions for $r=2,3$ and $4$.}\label{234}

In the present section we prove Conjecture \ref{bipconn} for the small cases in strongest possible form.

\begin{proposition}
\label{r=2} If a biclique $[X,Y]$ is partitioned into at most two
bi-equivalence graphs, then each has at most two (non trivial)
connected components.
\end{proposition}
\begin{proof} Assume on the contrary that $x_jy_j, j=1,2,3,$ are three edges from three
distinct connected components of $G_1$,  where $x_j\in X$ and $y_j\in Y$. Then the path
$(x_1,y_2,x_3,y_1)$ is in $G_2$, but the color of  $x_1y_1$ is not $2$. Hence $G_2$ is not
bi-equivalence graph, a contradiction.
\end{proof}

\begin{proposition}
\label{r=3}
Let a biclique $[X,Y]$ be partitioned into three  bi-equivalence graphs. If  one of those has
 more than three non trivial components, then some of the other two is spanning and has
 two connected components.
\end{proposition}
\begin{proof}
Assume on the contrary that $x_jy_j, j=1,2,3,4$, are four edges from four distinct connected
 components of $G_1$, where $x_j\in X$ and $y_j\in Y$.

The subgraph of the biclique on the vertex set $\{x_1,x_2,x_3,y_1,y_2,y_3\}$ contains a $6$-cycle $C$
whose edges are colored with $2$ and $3$. Since the color classes are bi-equivalence graphs, $C$ has
no monochromatic path of length more than two.

First assume  that $C$ has three edges of color $2$ (the other three
are colored with $3$). W.l.o.g. we assume that $x_4y_1,x_4y_2\in
E(G_2)$. By the bi-equivalence property, we have
$x_1y_2,x_2y_1\in E(G_3)$. Since $C$ has three edges in $G_2$, we
may assume $y_1x_3,y_2x_3\in E(G_2)$. Observe that the edges
$x_1y_3,x_2y_3$ of $C$ are colored differently from the set
$\{2,3\}$ hence the color of $x_4y_3$ is neither $2$ nor $3$, a
contradiction.

Therefore $C$ has four edges in one color and two edges in the other
color. W.l.o.g. we assume that the colors follow each other along
the cycle $C=(x_1,y_3,x_2,y_1,x_3,y_2,x_1)$ as $2,2,3,2,2,3$.  Then
for every vertex $x\in X\setminus (X_1[x_1]\cup X_1[x_2]\cup X_1[x_3])$ we
obtain that $xy_1,xy_2\in E(G_2)$. Observe that this is also true
for every $x\in X_1[x_3]$, since the $(2,3)$-coloring pattern along
the $6$-cycle $C^\prime=(C-x_3)+x$ uniquely determines the color
of the two edges at $x$.

In the same way one obtains that   $X\setminus (X_1[x_1]\cup X_1[x_2])$ and $Y_1[y_1]\cup Y_1[y_2]$
induce a biclique in $G_2$, since, for $i=1,2$,  any vertex $y\in Y_1[y_i]$ can replace  $y_i$
in the cycle $C$ without altering the
$(2,3)$-coloring pattern along the modified cycle. By symmetry of $X$ and $Y$, we obtain that
$Y\setminus (Y_1[y_1]\cup Y_1[y_2])$ and $X_1[x_1]\cup X_1[x_2]$ induce a biclique of $G_2$ as well.

Therefore $G_2$ is spanning and has two connected components.
\end{proof}

The propositions above imply immediately
\begin{corollary}\label{23width} For $r=2,3$, in any spanning partition of a biclique into $r$ bi-equivalence graphs some color class has width at most $r$.
\end{corollary}

With the antichain assumption Corollary \ref{23width} extends for $r=4$ as well:

\begin{theorem}
\label{r=4} If a biclique has an antichain partition into four
bi-equivalence graphs then the width of some color class is at most $4$.
\end{theorem}

\begin{proof} Let $G_i, i=1,2,3,4$, be the bi-equivalence graphs in a reduced antichain partition of a biclique $[X,Y]$.\\

\noindent{\it Claim 1:} if $|X_i[u]|\leq 2$ for every color $i$ and
vertex $u$, then $G_1$ has  $4$ components.\\
To see this let $y\in Y$ and set $U=\bigcup\limits_{i=2}^4 X_i[y]$.
Let $s$ be the number of components of $G_1$ that intersect $U $ at a
single vertex. If  $x\in X_i[y]$, for some $i\in\{2,3,4\}$, and
$B_1[x]\cap U= \{x\} $, then $X_1[x]= \{x\} $ and hence by the
antichain property,  $X_i[y]=\{x\}$ follows. Thus for the number of
components of $G_1$ different from $B_1[y]$ we obtain $s+2(3-s)/2=3,$
and the claim follows.\\

Due to Claim 1  we may assume that  there are three distinct
vertices, $x_1,x_2,x_3 \in X$ in some block of $G_1$. Let
$$Y(c_1,c_2,c_3)= \{y\in Y\mid yx_i\; \hbox{is colored with }\;c_i ,
i=1,2,3\}.$$
The three-tuple $(c_1,c_2,c_3)$ will be called the type of the subset
$Y(c_1,c_2,c_3)$. In terms of this notation $Y(1,1,1)\neq\emptyset$.
When the wildcard character $*$ is used for a color, then the color of
the corresponding edge between $\{x_1,x_2,x_3\}$ and the set of that
type is undetermined (e.g. $Y(3,3,4)\subseteq Y(3,*,4)$ is true).

In a bi-equivalence graph partition certain types cannot coexist as is
expressed in the next claim:\\

\noindent{\it Claim 2:} If $a,b$ are distinct colors, then at least
one of the sets $Y(a,a,*)$ and $Y(a,b,*)$ must be empty.
Indeed, if  $y_1\in Y(a,a,*)$ and $y_2\in Y(a,b,*)$, then
$(y_2,x_1,y_1,x_2)$ is a path belonging to some biclique of  $G_a$,
hence the edge $x_2y_2$ must have color $a$, and not $b$.

Using that $x_1$, $x_2$ are not equivalent we obtain the following\\

\noindent{\it Claim 3:} If $Y(2,2,*)\neq\emptyset$ then $Y(3,4,*)$ and
$Y(4,3,*)$ are not empty.\\

\noindent {\it Claim 4:} $Y(i,i,i)=\emptyset$ for every $i$ in $\{2,3,4\}$.
Assume on the contrary that $Y(2,2,2)\neq\emptyset$. Because $x_1$ and
$x_2$ are not equivalent,
we have $Y(3,4,*)\neq\emptyset$,  $Y(4,3,*)\neq\emptyset$, and
therefore,  $Y(3,3,*)=\emptyset$,  $Y(4,4,*)=\emptyset$. Moreover,
this must hold for any pair $x_i,x_j$, $1\leq i<j\leq 3$, which is
impossible (by the pigeon hole principle).\\

\noindent {\it Claim 5:} At least one of $Y(2,2,3)$ and $Y(2,2,4)$ is
empty. To see this, assume
$Y(2,2,3)\neq\emptyset$ and $Y(2,2,4)\neq\emptyset$.
By the previous claims we have $$Y=Y(1,1,1)\cup Y(2,2,3)\cup
Y(2,2,4)\cup Y(3,4,2)\cup Y(4,3,2),$$
where none of these types are empty.
In particular $Y(*,*,3)\cup Y(*,*,4)\subseteq Y(2,*,*)$, violating the
antichain property.\\

Now w.l.o.g. assume that either $Y(2,2,3)\neq\emptyset$ or in any
(nonempty) type $Y(a,b,c)$ the elements $a$, $b$, and $c$ are
distinct, apart $Y(1,1,1)$. In both cases every (nonempty) type in
$Y\setminus Y(1,1,1)$ has a color $3$. Then the components $B_3[x_i],
i=1,2,3,$ form a cover provided $Y_3[z]\cap (Y\setminus
Y(1,1,1))\neq\emptyset$, for all $z\in X$. If some $z$ does not
satisfy this, then by the antichain property, $Y(1,1,1)=Y_3[z]$, and
$B_3[x_i]$, $i=1,2,3$, and $B_3[z]$ together form a cover.
\end{proof}

\section{Bi-equivalence partitions for $r=5$.}\label{5}
In this section we shall verify Conjecture \ref{anticonj}, for
$r=5$, in a stronger form. Actually we will show that under the
appropriate conditions there is a cover with at most $2r-2=8$
monochromatic
 components in the same color,
or equivalently, one of the bi-equivalence graphs of the partition has width at most $8$.

\begin{theorem}
\label{r=5} If a biclique has an antichain partition into five
bi-equivalence graphs, then the width of some color class is at most $8$.
\end{theorem}

Let $G_i, i=1,2,3,4,5$, be the bi-equivalence graphs in a reduced antichain partition of the
biclique $[X,Y]$.  For the proof we need two technical lemmas.

\begin{lemma}
\label{singleton} If each $G_i, i=1,\dots,5,$ has width at
least $6$, then $[X,Y]$ contains at most two singletons in each vertex class.
\end{lemma}
\begin{proof}
Suppose on the contrary that one class has three singletons, say $x_1, x_2, x_3\in X$ with
$|X_i[x_j]|=1$, for every $1\leq i\leq 5, $ and $1\leq j\leq 3$.
Then taking any $y\in Y_1[x_1]$, we may assume that $yx_2\in E(G_2)$ and $yx_3\in E(G_3)$.
In particular, we obtain that $X=\{x_1, x_2, x_3\}\cup X_4[y]\cup X_5[y]$.

For any $z\in X_4[y]$, we have $X_5[z]\cap X_5[y]=\emptyset$, hence by the antichain property,
$X_5[z]=X_4[y]$. Therefore $G_5$ has five components: $B_5[x_1]$, $B_5[x_2]$, $B_5[x_3]$, $B_5[z]$,
$B_5[y]$, a contradiction.
\end{proof}

\begin{lemma}
\label{optimal} Let each $G_i, i=1,\dots,5,$ have width at
least $9$. If $[X,Y]$ contains at most two singletons in both of its vertex classes, then
there is a color $i$ and a vertex $u$ for which $|X_i[u]|\ge 9$ or
$|Y_i[u]|\ge 9$.
\end{lemma}

\begin{proof}
Assume that for every color $i$ and vertex
$u$ we have $|X_i[u]|\le t$ and $|Y_i[u]|\le t$. Let $G_1$ be the graph with the maximum number of
edges among $G_i,i=1,\dots,5$. The trivial inequality
$|E(G)|\le 5 |E(G_1)|$ will give us a first lower bound on $t$.

For a vertex $u\in X$ we have $Y=Y_1[u]\cup Y_2[u]\cup Y_3[u]\cup Y_4[u]\cup Y_5[u]$.
As $|Y_i[u]|\le t$ we get $|Y|\le 5t$. Similarly it follows that $|X|\le 5t$. Since $G$
contains at most two singletons, and the width of $G_1$ is at least $9$ we have
$5t\ge |Y|\ge 2\cdot 1+7\cdot2=16$, therefore $t\ge 4$.\\

Let $\x$ and $\y$ be vectors which contain the sizes of the components of $G_1$ in $X$ and in $Y$,
respectively. Our assumptions on $G_1$ mean that the length of $\x$ and $\y$ is at least $9$, they
have at most two elements equal to $1$, and all their elements are at most $t$.
Using this notation $|E(G_1)|=\x\cdot\y$, and $|E(G)|=|X||Y|=(\x\cdot\e)(\y\cdot\e)$, where $\e$ is
the constant $1$ vector with appropriate length.
We are going to investigate $\diff(\x,\y)=|E(G)|-5|E(G_1)|=(\x\cdot\e)(\y\cdot\e)-5(\x\cdot\y)$, and
determine its minimum over all possible values of $\x$ and $\y$. If this function is positive for some $t$,
then there is no partition of $G$ into graphs with the above conditions for the given value of $t$.\\

 In the first steps we minimize $\diff(\x,\y)$, for any fixed $|X|$ and $|Y|$, that is we
maximize $|E(G_1)|=\x\cdot\y$.\\

\noindent{\it Step 1:} We may assume that the length of $\x$ is equal to $9$, and so the length of $y$ is also $9$.
Otherwise we could join two components of $G_1$ and increase the number of edges.
So we have $\x=(x_1,\dots,x_9)$ and $\y=(y_1,\dots, y_9)$.\\

\noindent{\it Step 2:} We can reorder the components of $G_1$ such that $\y$ is ordered non-increasingly.
After that we may assume that the elements of $\x$ are also ordered non-increasingly.
Otherwise we could swap two elements with $x_i<x_j$ for $1\le i<j\le 9$ and this operation would not
decrease the value of $\x\cdot\y$. (The increment is $(x_j-x_i)(y_i-y_j)\ge 0$.)
Hence $y_1\ge y_2\ge \dots \ge y_9$ and $x_1\ge x_2\ge\dots \ge x_9$.\\

\noindent{\it Step 3:}
For $j>i$, the operation of increasing
$x_i$  and decreasing $x_j$  by the same constant $c$
 increases $|E(G_1)|=\x\cdot\y$ with $c(y_i-y_j)\ge 0$.\\
By repeated use of this operation (observing the condition that each
element of $\x$ and $\y$ is at most $t$, and these vectors contain
at most two elements equal to $1$) we obtain that $x_1=\dots
=x_p=t$, $t>x_{p+1}\ge 2$, $x_{p+2}=\dots =x_7=2$, $x_8=x_9=1$ and
similarly $y_1=\dots =y_q=t$, $t>y_{q+1}\ge 2$, $y_{q+2}=\dots
=y_7=2$, $y_8=y_9=1$. From $|X|\le 5t$
it follows that $p<5$, and similarly we get $q<5$.\\

Thus for a given $|X|$ and $|Y|$, the maximum value $|E(G_1)|=\x\cdot\y$ is determined by
the vectors $\x, \y$ standardized as above. In the next steps we minimize $\diff(\x,\y)$ by
changing $|X|$ and $|Y|$.\\

\noindent{\it Step 4:} If $x_{p+1}\neq 2$ then let $\x^{-}$ and $\x^{+}$ be
vectors almost the same as $\x$, but at the $(p+1)$-th position
they have $x_{p+1}-1\ge 2$ and $x_{p+1}+1\le t$, respectively. We
claim that $\diff(\x^{-},\y)$ or $\diff(\x^{+},\y)$ is not greater
than $\diff(\x,\y)$. Indeed,
$\diff(\x,\y)-\diff(\x^{-},\y)=\diff(\x^{+},\y)-\diff(\x,\y)=|Y|-5y_{p+1}$
which means that $\diff(\x,\y)$ is a middle element of an
arithmetic progression between $\diff(\x^{-},\y)$ and
$\diff(\x^{+},\y)$. Thus we may assume that $x_{p+1}=2$ and
similarly $y_{q+1}=2$. Furthermore we assume that $q=p+r$, where
$r\ge 0$.\\

\noindent{\it Step 5:} Now we can express $\diff(\x,\y)$ as a function of $p$ and $r$ in the following way:
$$
\begin{array}{lll}
\diff(\x,\y)&=&(\x\cdot\e)(\y\cdot\e)-5(\x\cdot\y) \\
&=& (tp+2(7-p)+2)(t(p+r)+2(7-p-r)+2)\\
&&-5(t^2p+2tr+4(7-p-r)+2),
\end{array}
$$
where the coefficient of $r$ is $p(t-2)^2+6(t-2)> 0$, as $t\ge 4$. Therefore $\diff(\x,\y)$ is minimal
 if $r=0$, that is $p=q$, and so $\x=\y$.
In this case $\diff(\x,\x)=p^2(t^2-4t+4)+p(-5t^2+32t-44)+106$, which has extremum if
$\frac{\mathrm{d}}{\mathrm{d}p} \diff(\x,\x) = 0$ which gives $p=\frac{5t^2-32t+44}{2(t^2-4t+4)}$.
(This extremum is a minimum since
$\frac{\mathrm{d^2}}{\mathrm{d}p^2} \diff(\x,\x) = 2(t^2-4t+4)=2(t-2)^2>0$, because $t\ge 4$.)

\medskip

From the above formula we get $p=1.5$, for $t=8$,  which gives that the minimum value of
$\diff(\x,\y)$ for any $\x$, $\y$ is at least $25>0$. (Actually the minimum is $34$ which is taken
on the integer values $p=1$ and $p=2$.) Thus $|E(G)|\le 5|E(G_1)|$ cannot hold for $t=8$, which
completes the proof. \end{proof}

\noindent {\it Proof of Theorem \ref{r=5}.}
Applying Lemmas \ref{singleton} and \ref{optimal},  it follows that there is a block containing
at least nine distinct  vertices, say $x_i\in X_1[x_1]$, for every $i=1,2,\dots,9$.
Similarly to the proof of Theorem \ref{r=4}, for a sequence of given colors $c_1,\dots, c_9$, let
$$Y(c_1,\dots,c_9)= \{y\in Y\mid yx_i\; \hbox{is colored with}\;c_i , i=1,\dots,9\}.$$
The nine-tuple $(c_1,\dots,c_9)$ will be called the type of the
subset $Y(c_1,\dots, c_9)\subseteq Y$. In terms of this notation
Lemmas \ref{singleton} and \ref{optimal} imply that
$Y(1,\dots,1)\neq\emptyset$. Again, when the wildcard character $*$ is
used for the $i$-th color position in a type, then the color of
the corresponding edges to $x_i$ are undetermined.

In a bi-equivalence graph partition certain types cannot coexist as is expressed in the next rule. \\

\noindent{\bf Type rule.}  If $a,b$ are distinct colors, then at least one of the sets
$Y(a,a,*,\dots,*)$ and $Y(a,b,*,\dots,*)$ must be empty.

Indeed, if  $y_1\in Y(a,a,*,\dots,*)$ and $y_2\in Y(a,b,*,\dots,*)$, then $(y_2,x_1,y_1,x_2)$ is a
path belonging to $G_a$, hence
the edge $x_2y_2$ must have color $a$, and not $b$.\\

Notice that the Type rule remains valid when permuting colors
and/or when relabelling the vertices $x_1,x_2,\dots,x_9$, that is
when the colors in the types are moved to different positions.
Thus, for instance, types $(*,5,*,\dots,*,3)$ and
$(*,3,*,\dots,*,3)$ cannot coexist.

We will need a simple corollary of the antichain property as follows:\\

\noindent{\bf Starring rule.} If $Y_c[w]\subseteq Y(c_1,\dots,c_9)$, for some  $w\in X,$ then
equality must hold because $ Y(c_1,\dots,c_9)\subseteq Y(c_1,*,\dots,*)=Y_{c_1}[x_1]$, in that case
we say that $w$ ``stars" the set $Y(c_1,\dots,c_9)$ in color $c$.
\\

In the sequel when we write ``w.l.o.g. we assume", we mean: ``by appropriately permuting the colors and
relabelling $x_1,x_2,\dots, x_9$ we may assume".\\

We shall proceed with investigating the partition of  $Y^\prime=Y\setminus Y(1,\dots,1)$ into
different types. Note that if
$Y(c_1,\dots,c_9)\subseteq Y^\prime$, then we have $c_i\neq 1$, for every $i=1,\dots,9$.\\

\noindent{\bf Distinguishing rule 1.} If $Y(2,2,*,\dots,*)\neq\emptyset$ and
$Y(3,3,*,\dots,*)\neq\emptyset$, then
$$Y(4,4,*,\dots,*)\cup Y(5,5,*,\dots,*)=\emptyset,$$
 furthermore, $$Y(4,5,*,\dots,*)\neq\emptyset,\quad  Y(5,4,*,\dots,*)\neq\emptyset.$$

To see this recall that no equivalent vertices exist in the coloring, in particular
$x_1,x_2$ must be distinguished by the components in colors $4$ and $5$. If
$ Y(4,4,*,\dots,*)\neq\emptyset,$ then by the Type rule, $B_i[x_1]= B_i[x_2]$ for every $i=1,2,3,4$,
implying $B_5[x_1]= B_5[x_2]$, hence $x_1,x_2$ would be equivalent.
An immediate corollary of Distinguishing rule 1 is stated for convenience as follows.\\

\noindent{\bf Distinguishing rule 2.} At least one of $Y(2,2,2,*,\dots,*)$ and $Y(3,3,3,*,\dots,*)$
must be empty.\\

Returning to the proof let $Y(c_1,\dots,c_9)\subseteq Y^\prime$. Since $c_i\in\{2,3,4,5\}$,
some color must repeat at least three times. We shall consider the following three cases:

1) there is a (nonempty) type in $Y^\prime$ such that a color repeats more than four times;

2) no (nonempty) type in $Y^\prime$ repeats a color more than four times, and there is a (nonempty) type repeating a color four times;

3) no (nonempty) type in $Y^\prime$ repeats a color more than three times.\\

\noindent {\it Case 1:} there is a (nonempty) type in $Y^\prime$ such that a color repeats more
than four times, say $Y(2,2,2,2,2,*,\dots,*)\neq\emptyset$.

Observe that color $2$ cannot repeat seven times. Indeed, in every (nonempty) type in $Y^\prime$
different from  $(2,2,2,2,2,2,2,*,*)$  color $2$ is not used on the first  seven positions, by the
Type rule. Hence one color among $3,4,$ and $5$ must repeat at least three times contradicting
Distinguishing rule 2. Thus we may assume that $Y(2,2,2,2,2,*,c_7,*,*)\neq\emptyset$, where
$c_7\neq 2$.

A similar pigeon hole argument shows that  in every (nonempty) type in $Y^\prime$ different from
$(2,2,2,2,2,*,*,*)$  each of the three colors $3,4,5$ must be used on the first  five positions,
otherwise Distinguishing rule 2 is violated. Thus w.l.o.g. we assume that $c_7=3$.

Observe that by the Type rule,
$Y_3[x_7]\subseteq Y(2,2,2,2,2,*,\dots,*)$, thus by the Starring rule, $Y_3[x_7]=Y(2,2,2,2,2,*,\dots,*)$
follows.
Then we obtain that
$$Y^\prime=\left(\cup\{Y_3[x_i]\mid 1\leq i\leq 5\}\right)\cup Y_3[x_7].$$

If  the six connected components $B_3[x_i], 1\leq i\leq 5$ and $B_3[x_7]$ do not cover $X$,
then there is an uncovered vertex $w\in X$ which stars $Y(1,\dots,1)$ in color $3$, by the
Starring rule. In this case
$B_3[x_i], 1\leq i\leq 5$, $B_3[x_7]$, and $B_3[w]$ cover $Y$ (thus the whole vertex set of $G$).

Consequently, in either case $G_3$ has width at most $7$.\\

\noindent {\it Case 2:} no (nonempty) type in $Y^\prime$ repeats a color more than four times, and
there is a (nonempty) type repeating a color four times, say $Y(2,2,2,2,c_5,\dots,c_9)\neq\emptyset$,
where $c_5,\dots,c_9\neq 2.$ We also know that among the five colors, $c_5,\dots, c_9$, there are
two distinct colors, w.l.o.g. we assume that $c_5=3$ and $c_6=4$.

Assume now that  in every (nonempty) type in $Y^\prime$ different from  $(2,2,2,2,*,\dots,*)$
color $3$ is used somewhere on the first  four positions. Then a similar argument that we used in
Case 1 shows that the width of  $G_{3}$  is at most $6$. By the same reason repeated for color $4$,
it remains to consider the situation when,  for each color $3$ and $4$, there is a (nonempty) type in
$Y^\prime$ different from  $(2,2,2,2,*,\dots,*)$ missing $3$ and $4$  on the first four positions,
respectively.

Since a color cannot repeat three times on the first four positions, we have that \\
$Y(4,4,5,5,*,\dots,*)\neq\emptyset$, moreover
$Y(a,b,c,d,*,\dots,*)\neq\emptyset$, where among $a,b,c,d$ both
colors $3$ and $5$ repeat twice. By the Type rule, either
$a=b=5, c=d=3$ or $c=d=5, a=b=3$. In each case Distinguishing rule 1 is violated.\\

\noindent {\it Case 3:} no (nonempty) type in $Y^\prime$ repeats a color more than three times.

Then by the pigeon hole principle, each (nonempty) type in $Y^\prime$ has a color repeated three times.
Furthermore, if a type uses just three colors, then each of its three colors is repeated exactly
three times.

Let $Y(c,c,c,*,\dots,*)\neq\emptyset$, for some $c=2,3,4,$ or $5$. If each (nonempty) type uses color $c$ at
some position, then either the connected components $B_c[x_i], 3\leq i\leq 9$ cover $X$, or some
$w\in X$ stars $Y(1,\dots,1)$ in color $c$,  hence $B_c[x_i], 3\leq i\leq 9$ and $B_c[w]$ cover $Y$
(thus the whole vertex set of $G$). In each situation $G_c$ has width at most $8$. We claim that this
must happen for some $c$.

Assume that color $2$ repeats three times in some (nonempty) type, and some other (nonempty) type misses color $2$.
W.l.o.g. let $T_2=(3,3,3,4,4,4,5,5,5)$ be a (nonempty) type. By repeating the same idea, we see that,
for every $c=3,4,5$, some (nonempty) type $T_c$ misses $c$.

Thus $T_3$ has three triplets in colors $2,4,5$ at some positions.   By Distinguishing rule 2 and the Type rule the last three positions of
$T_3$ cannot be $5,5,5$. W.l.o.g. assume that
$T_3=(5,5,*,5,*,\dots,*)$. Then again, by Distinguishing rule 2 and the Type rule, it follows that
$T_3=(5,5,4,5,2,2,4,4,2)$.

Finally, for the possible positions of the three 5's of $T_4$ with respect to $T_2$ and $T_3$, we
conclude as before that $T_4=(*,*,5,*,5,5,*,*,*)$.
 This contradicts Distinguishing rule 1 on positions $5$ and $6$ and completes the proof of Theorem \ref{r=5}. \qed

\section{The dual form, transversals of $r$-partite intersecting
hypergraphs.}\label{dual}

Conjectures \ref{Rysconn} and \ref{bipconn}  can be translated
into dual forms as conjectures about transversals of $r$-partite
$r$-uniform intersecting hypergraphs. To do that, one should
consider the $r$ partitions defined by the monochromatic connected
components of an $r$-colored complete or complete bipartite graph as
hyperedges over the vertex set and consider the dual of this
hypergraph. This approach already turned out to be very useful, for
example results of F\"uredi established in \cite{FU1}
can be applied. A survey  on the
subject is \cite{GYSUR}.

An $r$-uniform hypergraph $H$ is defined by a finite set $V(H)$
called the vertex set of $H$, and by a set $E(H)$ of r-sets of
$V(H)$ called edges of $H$. An $r$-uniform hypergraph $H$ is called
$r$-partite if  there is a partition $V(H)=V_1\cup\dots\cup V_r$
such that $|e\cap V_i|=1$, for all $i=1,\dots, r$ and $e\in E(H)$.
 A hypergraph $H$ is called {\it intersecting} if $e\cap f\neq \emptyset$ for any $e,f\in E(H)$.
 A set $T\subseteq V(H)$ is called a transversal of $H$ provided $e\cap T\neq\emptyset$,
 for all $e\in E(H)$; the minimum cardinality of a transversal of $H$ is the transversal number
 of $H$ denoted by $\tau(H)$.\\

  The dual
of Conjecture \ref{Rysconn} is Ryser's conjecture for intersecting hypergraphs in its usual form as follows:

\begin{conjecture}\label{Ryshyp} If $\cal{H}$ is an intersecting $r$-partite
hypergraph then $\tau({{\cal{H}}})\le r-1$. \end{conjecture}

There are infinitely many examples of intersecting $r$-partite hypergraphs with transversal number equal to $r-1$. Take a finite projective plane of order $q$, then truncate it by removing one point and the incident $q+1$ lines. The remaining lines taken as edges define an intersecting  $(q+1)$-partite hypergraph with transversal number equal to $q$. (Note that the truncated projective plane is the dual of an affine plane.)
A related question,  finding $f(r)$,
the minimum number of edges among intersecting $r$-partite
hypergraphs with transversal number at least $r-1$,  was addressed
in \cite{MSY}, where it was shown that $f(3)=3, f(4)=6,$ and $ f(5)=9$. \\

Concerning our biclique cover conjectures, the dual of a spanning partition of a complete bipartite graph
into $r$ bi-equivalence graphs gives two $r$-partite
hypergraphs, ${\cal{H}}_1,{\cal{H}}_2$ on the same vertex set such
that for every $h_1\in E({\cal{H}}_1), h_2\in E({\cal{H}}_2)$,
$|h_1\cap h_2|=1$ holds,  moreover at each vertex there is at
least one edge from both hypergraphs. We call such hypergraph
pairs $1$-cross intersecting. Then Conjecture \ref{bipconn} restated in claim B
reads as follows:

\begin{conjecture}
\label{quickproof} Let ${\cal{H}}_1,{\cal{H}}_2$ be a pair of $1$-cross
intersecting $r$-partite hypergraphs. Then
$\tau({\cal{H}}_1\cup {\cal{H}}_2)\le 2r-2$.
\end{conjecture}

 To illustrate the advantage of the dual formulation, here is a
 quick proof showing that $h(r)$ is bounded (although with a bound weaker
 than the one  in Theorem \ref{main}).

\begin{proposition}
\label{quickproof} Let ${\cal{H}}_1,{\cal{H}}_2$ be a pair of $1$-cross
intersecting $r$-partite hypergraphs. Then each
partite class contains at most ${2(r-1)\choose r-1}$ vertices.
\end{proposition}

\begin{proof} Let $v_1,\dots,v_p$ be the vertices of a partite class
of ${\cal{H}}_1,{\cal{H}}_2$. For each $v_i$ select $f_i^1\in E({\cal{H}}_1),
f_i^2\in E({\cal{H}}_2)$ such that $v_i\in f_i^1\cap f_i^2$, and set
$g_i=f_i^1\setminus\{v_i\}, h_i=f_i^2\setminus\{v_i\}$. Then the pairs
$(g_i,h_i)$ form a cross-intersecting $r-1$-uniform family (in fact a very
special one). It is well known (see Exercise 13.32 in \cite{LO}) that such hypergraphs have at most
${2(r-1)\choose r-1}$ edges.
\end{proof}

\end{document}